\documentclass[10pt, letterpaper]{article}
\usepackage{fullpage}
\usepackage{amsmath, amssymb,amsthm}
\usepackage[usenames,dvipsnames]{color}
\newtheorem{theorem}{Theorem}[section]
\newtheorem{lemma}[theorem]{Lemma}
\newtheorem{proposition}[theorem]{Proposition}
\newtheorem{corollary}[theorem]{Corollary}
\numberwithin{equation}{section}

\newenvironment{definition}[1][Definition]{\begin{trivlist}
\item[\hskip \labelsep {\bfseries #1}]}{\end{trivlist}}

\newenvironment{remark}[1][Remark]{\begin{trivlist}
\item[\hskip \labelsep {\bfseries #1}]}{\end{trivlist}}

\def\P{{\mathbb P}}        % probability
\def\E{{\mathbb E}}        % expectation
\def\Z{{\mathbb Z}}         % integers
\def\F{{\cal F}}                 % sigma field
\def\N{{\mathbb N}}       % naturals
\def\R{{\mathbb R}}       % reals
\def\T{{\mathbb T}}        % torus  
\def\1{{\mathbf 1}}         % indicator
\def\S{{\cal S}}               % finite set
\def\Var{{\mathbf {Var}\,}}    % variance
\def\deq{\stackrel{\scriptscriptstyle d}{=}}			   % equal in distribution
\def\Qset{{\mathcal Q}}      %set of products

\title{Universally $L^1$-Bad Arithmetic Sequences}
\author{Patrick LaVictoire\footnote{Supported in part by NSF Grant DMS-0401260.}, UC Berkeley}

\begin{document}
\maketitle

\begin{abstract}
\noindent We present a modified version of Buczolich and Mauldin's proof that the sequence of square numbers is universally $L^1$-bad.  We extend this result to a large class of sequences, including the $d$th powers and the set of primes; furthermore, we show that any subsequence of the averages taken along these sequences is also universally $L^1$-bad.
\end{abstract}

\section{Introduction}

Let $(X,\F,\mu,\tau)$ a dynamical system.  For a sequence of natural numbers $\{n_k\}$ and any $f\in L^1(X)$, we can consider the subsequence average
$$A_Nf(x):= \frac1N \sum_{k=1}^N f(\tau^{n_k}x).$$
By analogy with Birkhoff's Pointwise Ergodic Theorem, we will examine the a.e. convergence or divergence of $A_Nf(x)$ as $N\to\infty$.
\\ \\
We say that $\{n_k\}$ is \emph{universally $L^p$-good} if for every dynamical system $(X,\F,\mu,\tau)$ and every $f\in L^p(X,\mu)$, $\displaystyle\lim_{N\to\infty} A_N f(x)$ exists for almost every $x\in X$.  We say that $\{n_k\}$ is \emph{universally $L^p$-bad} if for every non-atomic ergodic dynamical system $(X,\F,\mu,\tau),$ there exists an $f\in L^p(X,\mu)$ such that the sequence $\{A_N f(x)\}_{N=1}^\infty$ diverges on a set of positive measure in $X$.  Finally, we say that $\{n_k\}$ is \emph{persistently universally $L^p$-bad} if for every non-atomic ergodic dynamical system $(X,\F,\mu,\tau)$ and every infinite $\S\subset\N$, there exists an $f\in L^p(X,\mu)$ such that the sequence $\{A_N f(x)\}_{N\in\S}$ ($N$ taken in increasing order) diverges on a set of positive measure in $X$.
\\ \\
Among the classical results in this topic, Bourgain \cite{JB3} proved that (the integer part of) any sequence of polynomial values is universally $L^p$-good for any $p>1$, and Bourgain \cite{APET} and Wierdl \cite{Primes} showed that the same is true of the sequence of prime numbers.  For these sequences, the Banach principle of Sawyer \cite{Sawyer} implies that pointwise convergence of $A_N f$ for all $f\in L^1$ depends only on the validity of a weak maximal inequality 
\begin{eqnarray}
\label{confusion}
\|\sup_{N\in\N} |A_Nf(x)|\|_{1,\infty}\leq C\|f\|_1\qquad \forall f\in L^1(X).
\end{eqnarray}
The Conze principle \cite{Conze} allows the transfer of such an inequality (with the same constant) from any ergodic dynamical system $(X,\F, \mu,\tau)$ to any other dynamical system.  Therefore, one of these sequences would be universally $L^1$-good if and only if there were some fixed $C>0$ such that (\ref{confusion}) held for every dynamical system $(X,\F,m,T)$, and it would be universally $L^1$-bad otherwise.
\\ \\
This question of pointwise convergence for subsequence averages of $L^1$ functions remained open for virtually all sequences of interest, including all polynomials of degree $\geq2$ and the sequence of primes, until Buczolich and Mauldin \cite{DAAS} \cite{DSA} proved that $\{k^2\}$ is in fact universally $L^1$-bad. 
\\ \\
In this paper, we adapt and extend the construction in \cite{DAAS} to prove the following theorem:
\begin{theorem}
\label{dylan}
Let $n_k=k^d$ for some $d>1$, or $n_k=$ the $k$th prime number.  Given any $C>0$ and any infinite set $\S\subset \N$, there exists a dynamical system $(X,\F,\mu,\tau)$ and an $f\in L^1(X)$ such that $$\|\sup_{N\in\S} |\frac1N \sum_{k=1}^N f\circ\tau^{n_k}|\|_{1,\infty}>C\|f\|_1.$$
\end{theorem}
\noindent As discussed above, this implies
\begin{corollary}
The sequence of $d$th powers ($d>1$) and the sequence of primes are persistently universally $L^1$-bad.
\end{corollary}

\noindent This exposition is self-contained, with the exception of the number-theoretic results of Hooley \cite{Hooley} and Granville and Kurlberg \cite{Poisson}.
\\ \\
A few words on the structure of this paper: In Section 2, we present a heuristic version of the argument, in the case of the squares. Then in Section 3, we express the general form of our result (Theorem \ref{jimi}) and prove that its conditions are indeed satisfied by the $d$th powers and the sequence of primes.  In Section 4, we present the main inductive step (Proposition \ref{KML}), show that it implies Theorem \ref{jimi}, and explain the structure of the induction.
\\ \\
In Sections 5-7, we construct the various objects of the succeeding inductive step and prove several necessary lemmas about them.  Section 8 brings these parts together and proves that the properties claimed in Proposition \ref{KML} do indeed hold for this next step, completing the proof of Theorem \ref{jimi}.  In Section 9, we retrospectively explain the purpose of several objects and lemmas in this intricate proof.
\\ \\
Our notation will rarely distinguish between $\Z_N$ (a probability space with the measure-preserving transformation $\tau x=x+1 \mod N$) and $\Z$.  Sets and functions on $\Z_N$ will correspond to $N$-periodic sets and functions on $\Z$, and any object on $\Z_N$ is understood to represent an object on $\Z_{MN}$ for any $M\in\Z^+$.
\\ 
\\Furthermore, we let $\P$ denote the uniform probability measure on $\Z_N$, and $\E X$ the expected value of a random variable $X:\Z_N\to\R$.  Note that the values of $\P$ and $\E$ are unchanged when we consider a $N$-periodic set or function as an object on $\Z_{MN}$ instead, so that we may use $\P$ and $\E$ freely without keeping track of $N$.  We will use $X\deq Y$ to denote that two random variables $X$ and $Y$ (not necessarily on the same probability space) have identical distributions.
\\
\\Finally, we will use both subscripts and superscripts on certain functions $f_h^L, g_h^L, X_h^L, Z_h^L$ and certain sets $\Lambda_q^\gamma,\Xi_q^\gamma$.  To prevent these from being confused with exponential notation, we note here that such superscripts on these objects will not denote exponents; we will therefore write the square of $X_h$ as $(X_h)^2$ rather than $X_h^2$.

%Outline Section
\section{Outline of the Argument for the Squares}
Here we will present a heuristic outline of the argument in the original case $n_k=k^2$, before introducing the necessary complications (exceptional sets and the like).  We therefore ask the reader's patience with these claims, some of which are not technically true; the argument presented in Section 3 and thereafter will be rigorous.
\\
\\Buczolich and Mauldin's proof in \cite{DAAS} essentially boils down to three key insights.  The first is that in order to prove that $\{k^2\}$ is universally $L^1$-bad, it suffices to prove the existence of what they term $(K,M)$ families\footnote{Buczolich and Mauldin define these families with many more conditions, in the fashion that we will define our inductive Step $(K,M,L)$ in Proposition \ref{KML}; but the definition here will suffice for heuristic purposes.} for arbitrary $K,M\in\N$.
\begin{definition}
Given $K,M\in\N$ and a measure-preserving system $(\Omega,\F,\tau,\P)$, a \emph{$(K,M)$ family on $\Omega$} consists of the following:
\begin{itemize}
\item $f_1, \dots,f_K\in L^1(\Omega)$ with $f_h\geq0$ and $\E f_h\leq1$
\item $X_1,\dots,X_K\in L^1(\Omega)$ pairwise independent with $\E X_h\geq M$ and $\E (X_h)^2\leq C_M$
\item A measurable function $Q_x:\Omega\to\N$ such that for a.e. $x\in\Omega$,
\begin{eqnarray}
\label{reason}
\frac1{Q_x}\sum_{k=1}^{Q_x} f_h(\tau^{k^2}x)\geq X_h(x)\qquad\forall 1\leq h\leq K
\end{eqnarray}
\end{itemize}
\end{definition}
Note that for each $x\in\Omega$, $Q_x$ does not depend on $h$.
\\
\\The point of constructing a $(K,M)$ family is that, while a single $X_h$ may have a weak $L^1$ norm no greater than the $L^1$ norm of $f$, an average of pairwise independent random variables  with uniformly bounded variance is subject to the Weak Law of Large Numbers.  Thus for some large $K$, the average $\frac{X_1(x)+\dots+X_K(x)}{K}$ will be at least $\frac{M}2$ on a set of probability at least $\frac12$, giving a weak $L^1$ norm of at least $\frac{M}4$, while the $L^1$ norm of the average $\frac{f_1(x)+\dots+f_K(x)}{K}$ remains $\leq1$.
\\
\\ Therefore if we have $(K,M)$ families for all $K,M\in\N$, we can construct a dynamical system and a function that violate any given weak (1,1) maximal inequality.
\\
\\ The second insight is that we can inductively construct a $(K,M)$ family on the probability space $\Z_N$ with the shift operator $\tau x=x+1\mod N$, for some large $N$ depending on $K$ and $M$.  We will use an inner induction: given a $(K-1,M)$ family, we will construct $f_K$ and $X_K$ in stages $f_K^L,\;X_K^L$, so that we have all the required properties except for $\E X_K^L\geq M$, which we replace with $\E X_K^{L+1}\geq\E X_K^L+\epsilon$.  Then for $L$ large enough, we will have a $(K,M)$ family as desired.
\\ \\The main obstruction to this approach is the difficulty of maintaining pairwise independence; when we alter $f_K^L$ and $X_K^L$, we must at the same time alter $f_1^L,\dots,f_{K-1}^L$ and $X_1^L,\dots,X_{K-1}^L$ in order to sustain this property.  We do this by taking into account properties of the distribution of the squares in residue classes.

%Lambda
\begin{definition}
Given $\gamma>0$, and given $q\in\N$ squarefree and odd with $\kappa$ prime factors, consider the following subsets of the integers:
\begin{eqnarray*}
\Lambda_q&:=&\{x\in\Z:(x,q)=1,\;\exists k\in \Z \text{ such that }x\equiv k^2 \mod q\}\\
(-\Lambda_q)^\gamma&:=& -\Lambda_q+(0,\gamma2^\kappa)
\end{eqnarray*}
\end{definition}
It follows from elementary number theory that every $x\in\Lambda_q$ has exactly $2^\kappa$ square roots modulo $q$; and that if $q=p_1\dots p_\kappa$ with $p_i$ large, then $\P(\Lambda_q)=\prod_{i=1}^\kappa \frac{p_i-1}{2p_i}\approx 2^{-\kappa}$.  Clearly $\P((-\Lambda_q)^\gamma)\leq\gamma$; it follows from the result of Granville and Kurlberg \cite{Poisson} that $\P((-\Lambda_q)^\gamma)\approx 1-e^{-\gamma}$.
\\
\\Having set $f_K^0\equiv X_K^0\equiv1$, at each step in $L$ we will take some highly composite $q=q_L$ and set
\begin{eqnarray}
\label{life}
f_K^{L+1}(x)&:=&e^\gamma f_K^L(x)\1_{\Z\setminus(-\Lambda_q)^\gamma}(x)\\
\label{joke}
X_K^{L+1}(x)&:=&e^\gamma X_K^L(x)\1_{\Z\setminus(-\Lambda_q)^\gamma}(x)+c\1_{\Psi_q}(x),
\end{eqnarray}
for some $\Psi_q\subset(-\Lambda_q)^\gamma$ with $\P(\Psi_q)\geq c\gamma$.  To maintain the property (\ref{reason}), for $x\not\in(-\Lambda_q)^\gamma$ we will keep the same length of averaging $Q_x$ we used for $f_K^L$; but for $x\in(-\Lambda_q)^\gamma$ we will redefine $Q_x$ to be a multiple of $q$.  We will define the set $\Psi_q$ such that if we add a generic square number to an element of $\Psi_q$, we have a good chance of landing in the support of $f_K^{L+1}$; the result of \cite{Poisson} ensures that this is possible.
\\
\\ We can ensure that $X_K^{L+1}$ remains independent from our original $(K-1,M)$ family by choosing $q$ relatively prime to the period of that family; note that $\E X_K^{L+1}=\E X_K^L+c^2$.  However, we are not quite out of the woods: for $x\in(-\Lambda_q)^\gamma$ we are now using different values of $Q_x$ with $f_K^{L+1}$ than with our $(K-1,M)$ family.
\\
\\ This defect is repaired by the third insight of \cite{DAAS}: using the inductive hypothesis, we take a different $(K-1,M)$ family $\{g_1^L,\dots, g_{K-1}^L, Z_1^L,\dots, Z_{K-1}^L,Q'_x\}$, where the lengths of averaging $Q'_x$ are relatively prime to $Tq$; we may then ``restrict'' it to $(-\Lambda_q)^\gamma$ by taking 
\begin{eqnarray*}
\bar g_h^L(x)&:=&q2^{-\kappa} g_h^L(x)\1_{(0,\gamma2^\kappa)+q\Z}(x)\\
\bar Z_h^L(x)&:=&Z_h^L(x)\1_{(-\Lambda_q)^\gamma}(x)\\
\bar Q'_x&:=&q Q'_x.
\end{eqnarray*}
This preserves the property (\ref{reason}) on $(-\Lambda_q)^\gamma$ if we assume that
\begin{eqnarray}
\label{excited}
g_h^L(\tau^{Q'_x}y)=g_h^L(y+Q'_x)=f(y)\qquad \forall y\in[x,x+(qQ'_x)^2),
\end{eqnarray}
since if $x\in(-\Lambda_q)^\gamma$, then $|\{k\in[1,qQ'_x]:x+k^2\in(0,\gamma2^\kappa)+q\Z\}|\geq2^\kappa Q'_x$ and this set is equidistributed modulo $Q'_x$.  This implies that for $x\in(-\Lambda_q)^\gamma$,
\begin{eqnarray*}
\frac1{\bar Q'_x}\sum_{k=1}^{\bar Q'_x} \bar g_h^L(x+k^2)=\frac1{qQ'_x}\sum_{k=1}^{Q'_x}2^\kappa\cdot q2^{-\kappa} g_h^L(x+k^2)\geq Z_h(x)=\bar Z_h(x).
\end{eqnarray*}
(That is, averages of $\bar g_h^L$ over long intervals of the squares look like averages of $g_h^L$ on short intervals of the squares.)  Then if for $h\leq K-1$ we let 
\begin{eqnarray*}
f_h^{L+1}(x):=f_h^L(x)\1_{\Z\setminus(-\Lambda_q)^\gamma}(x)+\bar g_h^L\\
X_h^{L+1}(x):=X_h^L(x)\1_{\Z\setminus(-\Lambda_q)^\gamma}(x)+\bar Z_h^L
\end{eqnarray*}
and set the new $Q_x$ to equal $\bar Q'_x$ on the set $(-\Lambda_q)^\gamma$, we find that we have nearly preserved the properties of the family we began with.  Thus we may iterate the inner inductive step (to which we must of course add a version of (\ref{excited})) and thereby construct a $(K,M)$ family.

%SECTION 3
\section{Main Theorem}

\begin{definition}
For a set $\Lambda\subset \Z_t$, we can define
\begin{eqnarray*}
\P(\Lambda)&:=&\frac{|\Lambda|}t\\
s_t&:=&\P(\Lambda)^{-1}\\
\Lambda^\gamma&:=&\Lambda+(0,\gamma s_t)\subset \Z_t.
\end{eqnarray*}
\end{definition}
\begin{remark}
Note that this turns $\Z_t$ into a probability space, that $s_t$ is the average spacing between elements of $\Lambda$ in $\Z_t$, and that $\P(\Lambda^\gamma)\leq\gamma$.  If $\Lambda\subset\Z$ is periodic by $t$, we can consider it as a subset of $\Z_{nt}$ for any $n\in\Z^+$, and we see that $\P(\Lambda)$ is independent of $n$.
\end{remark}
The main result of this paper is the following:
%MAIN THEOREM
\begin{theorem}
\label{jimi}
Let $\{n_k\}\subset\N$ an increasing sequence, and $\alpha,\beta>0$.  Say that for every integer $t>1$, there exists a set of residues $\Lambda_t\subset\{n_k + t\Z:k\in\N\}\subset \Z_t$ such that $\Lambda_{st}=\Lambda_s\cap\Lambda_t$ whenever $(s,t)=1$, and that there exist some auxiliary sequences $\{p_j\},\;\{q_j\}$ of pairwise relatively prime positive integers such that
\begin{eqnarray}
\label{wind}
\P(\Lambda_{q_j})\to0,\\
\label{howl}
\inf_j \P(\Lambda_{p_j})>0,\\
\label{wine}
\varepsilon_\gamma:=\gamma-\liminf_{j\to\infty}\P(\Lambda^\gamma_{q_j})&= &o(\gamma),
\end{eqnarray}
such that for all $\gamma>0$ sufficiently small,
\begin{eqnarray}
\label{earth}
\liminf_{j\to\infty}\left|\left\{(u,v)\in\Lambda_{q_j}\times\Lambda_{q_j}:|u-v-w|>\gamma s_{q_j}\;\forall w\in\Lambda_{q_j}\right\}\right|&> & 5\alpha|\Lambda_{q_j}|^2,
\end{eqnarray}
and for all $Q=p_{i_1}\dots p_{i_k}q_{j_1}\dots q_{j_l}$ with $ i_1<\dots<i_k, j_1<\dots<j_l$,
\begin{eqnarray}
\label{line}
\displaystyle \liminf_{N\to\infty} \frac1N \left|\left\{ 1\leq k \leq N: n_k\equiv a \mod Q\right\}\right|> \frac{\beta}{|\Lambda_Q|}& &\forall a\in\Lambda_Q.
\end{eqnarray}
Then, given any $C>0$ and any infinite set $\S\subset \N$, there exists a probability space $(X,\F,\P)$, a measure-preserving transformation $\tau$ on $X$, and an $f\in L^1(X)$ such that $$\|\sup_{N\in\S} |\frac1N \sum_{k=1}^N f\circ\tau^{n_k}|\|_{1,\infty}>C\|f\|_1.$$
\end{theorem}
\begin{remark}
(\ref{wine}) states that $\Lambda_{q_j}$ does not cluster too much in $\Z_{q_j}$ so that $\P(\Lambda_{q_j}^\gamma)$ is nearly $\gamma$, while (\ref{earth}) states that the set of differences of elements in $\Lambda_{q_j}$ is not concentrated near $\Lambda_{q_j}$.  (\ref{line}) states that each point of $\Lambda_Q$ is hit uniformly often by the sequence $\{n_k\}$, if $Q$ is any squarefree product of terms from the auxiliary sequences $\{p_j\}$ and $\{q_j\}$.  (The function $x\to Q_x$ we will later construct will take such products as its values.)
\end{remark}
\noindent \textbf{Claim.} \emph{Theorem \ref{jimi} implies Theorem \ref{dylan}.}

\begin{proof}
For the sequence of primes, we take $\Lambda_t$ to be the set of integers relatively prime to $t$, then let $p_j$ be distinct primes and $q_j$ be highly composite such that $q_j^{-1}\phi(q_j)\to0$; since $\displaystyle \prod_{p \text{ prime}} \frac{p-1}p =0$, it is possible to choose such $q_j$.  The property (\ref{line}) is clear from Dirichlet's theorem on arithmetic progressions.
\\ \\ For the sequence of $d$th powers, we will take $\Lambda_t$ to be the residues of $k^d$ mod $t$ which are units mod $t$.  We will take $p_j$ to be distinct primes congruent to 1 mod $d$, and $q_j$ to be products of $j$ such primes.  Note that if $Q$ has $\kappa$ prime factors each congruent to 1 mod $d$, then every $x\in\Lambda_Q$ will have precisely $d^\kappa$ $d$th roots in $\Z_Q$, and $\P(\Lambda_Q)= \phi(Q)d^{-j}$.  If we choose all of the prime factors sufficiently large, we can ensure that $\phi(Q)\geq\frac12$ for all squarefree products $Q$ of these sequences, so that (\ref{line}) is satisfied. 
\\ \\
To prove (\ref{wine}) and (\ref{earth}) for each of these cases, we will use some recent results on the distribution of the residues $\Lambda_{q_j}$ (as the average spacing $s_{q_j}\to\infty$); roughly speaking, in each case they are distributed locally like a Poisson process of rate $\P(\Lambda_{q_j})$.  We begin by introducing some notation.

\begin{definition}
Let $\{y_k\}_{k=1}^N$ be a strictly increasing sequence of real numbers in $[0,N)$, and consider it as a subset of $\T=\R/N\Z$.  For $E\subset \{1,\dots,N\}$, and $\theta>0$, we define a probability measure and a cumulative distribution function:
\begin{eqnarray}
\tilde\P_k(E)&:=&\frac{|E|}N\\
\label{barefoot}
F(\theta)&:=&\tilde\P_k( |y_{k+1}-y_k|>\theta).
\end{eqnarray}
\end{definition}
Now if we consider the normalized set $s_q^{-1}\Lambda_q=\{y_i: 1\leq i\leq |\Lambda_q|\}\subset[0,|\Lambda_q|)$ (where the $y_i$ are taken in increasing order), we can examine the cumulative distribution function $F_q(\theta)$ defined as above.  Hooley \cite{Hooley} proves for the primes, and Granville and Kurlberg \cite{Poisson} prove for the $d$th powers\footnote{While Corollary 2 in that paper is stated for the entire set of $d$th power residues modulo $q$, we see that passing to the subset $\Lambda_q$ does not change equation (1) and thus we may apply the same result.}, that $F_q(\theta)\to e^{-\theta}$ pointwise as $s_q\to\infty$ (for which reason they call these sets ``Poisson distributed'').  Thus
\begin{eqnarray*}
\P(\Lambda^\gamma_q)=|\Lambda_q|^{-1}\sum_i |y_{i+1}-y_i|\wedge\gamma =\int_0^\gamma F_q(\theta) d\theta\to 1-e^{-\gamma},
\end{eqnarray*}
so that we have (\ref{wine}).  To prove (\ref{earth}) from this distributional fact, however, requires a little more work.
% Distribution Lemma
\begin{lemma}
\label{poisson}
Let $\zeta_l(\theta):=\tilde\P_i(|y_i-y_k-y_l|>\theta \,\forall k)$.  For $\theta>0$ and $J>4$,
\begin{eqnarray*}
\tilde\P_l\left(\zeta_l(\theta)>\frac{F(J\theta)+F(2\theta)-1}2\right)&\geq & F(2\theta)-F((J-2)\theta).
\end{eqnarray*} 
\end{lemma}
\begin{proof}
Define $A:=\{k:|y_{k+1}-y_k|\leq J\theta \};$ then $\tilde\P_k(A)=1-F(J\theta)$.
\\
\\For $z\in[0,N)$, let 
\begin{eqnarray*}
B_z&:=&\{i: \exists k\in A: |y_i-y_k-z|\leq\theta\}\\
C_z&:=&\{i: \exists k\not\in A: |y_i-y_k-z|\leq\theta \}.
\end{eqnarray*}
There are $|B_z|$ points $y_i-z$ contained in the $|A|$ intervals $[y_k-\theta,y_k+\theta]$; thus it must be that for at least $|B_z|-|A|$ of these points, their successors $y_{i+1}-z$ also lie in the same interval, which implies that $|y_{i+1}-y_i|\leq2\theta$.  Thus $\tilde\P_i(B_z)\leq \tilde\P_k(A)+1-F(2\theta)=2-F(J\theta)-F(2\theta)$.
\\
\\ Therefore, for any $l$, 
\begin{eqnarray*}
1-\zeta_l(\theta)&\leq& \tilde\P_i(\exists k\not\in A: |y_i-y_l-y_k|\leq\theta)+\tilde \P_i( \exists k\in A: |y_i-y_k-y_l|\leq\theta )\\
&=&\tilde\P_i(C_{y_l})+\tilde \P_i(B_{y_l})\\
&\leq&\tilde\P_i(C_{y_l})+2-F(J\theta)-F(2\theta).
\end{eqnarray*}
But if $i\in C_{y_l}$ and $2\theta<|y_l-y_{l-1}|<(J-2)\theta$, then
$$y_i-y_{l-1}\in(y_k+\theta, y_{k+1}-\theta)$$ so that $|y_i-y_k-y_{l-1}|>\theta \,\forall k$.  This implies that $$\zeta_l(\theta)+\zeta_{l-1}(\theta)\geq \zeta_l(\theta)+\tilde \P_i(C_{y_l})\geq F(J\theta)+F(2\theta)-1$$ whenever $y_l-y_{l-1}\in(2\theta,(J-2)\theta)$.
\\
\\Therefore
\begin{eqnarray*}
\tilde\P_l\left(\zeta_l(\theta)>\frac{F(J\theta)+F(2\theta)-1}2\right)&\geq&\frac12\tilde\P_l(\zeta_l(\theta)+\zeta_{l-1}(\theta)>F(J\theta)+F(2\theta)-1)\\
&\geq& \tilde\P_l(2\theta<y_l-y_{l-1}\leq(J-2)\theta)\\
&\geq& F(2\theta)-F((J-2)\theta).
\end{eqnarray*}
\end{proof}

\noindent Now if $\gamma<\frac1{12}$, then $e^{-2\gamma}>\frac67$ and we can set $J:=\frac{\log 2}\gamma>4$; then Lemma \ref{poisson} implies
\begin{eqnarray*}
\tilde\P_l(\zeta_l>1/8)&\geq&\tilde\P_l\left(\zeta_l>\frac{-1+e^{-J\gamma}+e^{-2\gamma}}2+o_j(1)\right)\\
&\geq& e^{-2\gamma}-e^{(2-J)\gamma}+o_j(1)\geq\frac14
\end{eqnarray*}
 for $j$ sufficiently large, so that (\ref{earth}) holds with $\alpha=\frac1{160}$.
 \end{proof}
 
\begin{remark}
Note that (\ref{line}) cannot be satisfied by polynomials other than $n_k=c_dk^d+c_0$, as can be seen from Cohen's result \cite{Cohen} that if we fix $a\in\Z_p[x]$ of degree $d$ which is not of this type and consider $a(x)-y$ as $y\in\Z_p$ varies, for some fixed proportion of $y$ this polynomial will have $d$ distinct roots, while for some fixed proportion it will have 1 root.  This prevents (\ref{line}) from holding for sufficiently composite products.  However, the author expects that every polynomial of degree 2 or greater over $\Z$ should be persistently universally $L^1$-bad, and that some clever variant of this argument should suffice to prove as much.
\end{remark}

%Inductive Argument Begins
\section{The Inductive Step}

For an inductive argument to work, we will have to specify additional properties of the objects we seek, including a fixed distribution for the functions $X_h^L$.

\begin{definition}
Given the sequences $\{p_j\}$ and $\{q_j\}$ as in Theorem \ref{jimi}, denote the set of their (squarefree) products
\begin{eqnarray}
\Qset&:=&\{p_{i_1}\dots p_{i_k}q_{j_1}\dots q_{j_l}: i_1<\dots<i_k, j_1<\dots<j_l\}
\end{eqnarray}
and the functions (depending on the infinite $\S\subset\N$ in Theorem \ref{jimi})
\begin{eqnarray}
N(Q)&:=& \max\left\{N: \exists a\in\Lambda_Q: \left|\left\{ 1\leq k \leq N: n_k\equiv a \mod Q\right\}\right|\leq \frac{\beta N}{|\Lambda_Q|} \right\}\\
\psi(n)&:=& \inf\{s\in\S: s>N(Q) \;\;\forall Q\in\Qset, Q\leq n\}.
\end{eqnarray}
\end{definition}
\begin{definition}
For $0<\gamma<1$ and $0<\alpha<1$, we recursively define functions $Y_{n,\gamma,\alpha}:[0,1]^{n+1}\to\R$, which we may consider as random variables with the Lebesgue measure.  Let $Y_{0,\gamma,\alpha}(x_0)\equiv1$, and
\begin{eqnarray}
\label{distance}
Y_{n+1,\gamma,\alpha}(x_0,\dots,x_{n+1}):=(1-\gamma)^{-1}Y_{n,\gamma,\alpha}(x_0,\dots,x_n)\1_{[\gamma,1)}(x_{n+1})+\alpha\1_{[0,\alpha\gamma)}(x_{n+1}).
\end{eqnarray}
\end{definition}
Note that 
\begin{eqnarray*}
\E(Y_{n,\gamma,\alpha})&=&\E Y_{n-1,\gamma,\alpha}+\alpha^2\gamma=1+n\alpha^2 \gamma,\\ 
\E(Y_{n,\gamma,\alpha})^2&=&(1-\gamma)^{-1}\E (Y_{n-1,\gamma,\alpha})^2+\alpha^3\gamma\leq (1-\gamma)^{-n}(1+\alpha^3)\leq2(1-\gamma)^{-n}.
\end{eqnarray*}
\noindent Given the conditions of Theorem \ref{jimi}, we may assume that all the $p_j$ and $q_j$ are odd and that $\alpha$ is a dyadic rational.  Choose $\gamma_0<1/2$ such that for all $0<\gamma<\gamma_0$, (\ref{wine}) and (\ref{earth}) hold and $\varepsilon_\gamma<\alpha\gamma$; we will write $\varepsilon=\varepsilon_\gamma$ from now on unless otherwise specified.  We are now ready to state the inductive step:

%MAIN INDUCTIVE PROPOSITION
\begin{proposition}[Step $(K,M,L)$]
\label{KML}
Assume the conditions of Theorem \ref{jimi}.  Given any dyadic rational $0<\gamma<\gamma_0$, any $K,M,L\in\N$ with $L\leq M$, constants $A,\delta>0$ with $\delta<8\varepsilon$, and an odd integer $D$, there exist $T\in\Qset$ and $R\in\N$ with $(T,D)=(R,D)=1$ and the following objects:
\begin{description}
\item[(1)] $f_1, \dots,f_K\in\ell^1(\Z_T)$ with $f_h\geq0$ and $1\leq\E f_h\leq(1+4\varepsilon)^{(K-1)M+L}$.
\item[(2)]  $X_1,\dots,X_K\in\ell^1(\Z_{RT})$ pairwise independent with $X_h\deq Y_{M,\gamma,\alpha}\;\forall h<K$ and $X_K\deq Y_{L,\gamma,\alpha}$
\item[(3)] An exceptional set $E\subset \Z_T$ with $\P(E)\leq \delta$
\item[(4)] $Q_x:\Z_{T}\to\Z^+$ with $Q_x\in\Qset$ and $Q_x\mid T\;\forall x$, such that for each $x\not\in E$,
\begin{eqnarray}
\label{joker}
\frac1{|\Lambda_{Q_x}|}\sum_{\substack{a\in\Lambda_{Q_x},\\ 1\leq a\leq Q_x}} f_h(x+a)\geq X_h(x)&&\forall 1\leq h\leq K
\end{eqnarray}
and
\begin{eqnarray}
\label{thief}
f_h(x+y-Q_x)=f_h(x+y)\qquad \forall1\leq h\leq K,\; Q_x\leq y\leq \psi(AQ_x).
\end{eqnarray}
\end{description}
\end{proposition}
\begin{remark}
In the final section of this paper, we will discuss the significance of the parameters $A,D,R,$ and $\delta$, as well as the distribution $Y_{n,\gamma,\alpha}$, the reason we require $\gamma$ and $\alpha$ to be dyadic rationals, and other points whose necessity in the argument is not immediately obvious.  For the time being, we ask the reader's trust that these complications are required in order to make a strong enough inductive step.
\end{remark}

%(K,M,L) Is Enough
\noindent \textbf{Claim.} \emph{Proposition \ref{KML} implies Theorem \ref{jimi}.}
\begin{proof}
If we fix $C>0$, take $\gamma>0$ small, and set $M= \lfloor C/\gamma\rfloor$ and $K=\lfloor \gamma/C\varepsilon\rfloor$, and take $\delta\leq\frac14, A=D=1$, then at Step $(K,M,M)$ we have 
\begin{eqnarray*}
\E f_h&\leq& (1+4\varepsilon)^{KM}\leq (1+4\varepsilon)^{1/\varepsilon}\leq e^4, \\
\E X_h&=&\E Y_{M,\gamma,\alpha}\geq1+M\alpha^2\gamma\geq C\alpha^2,\\
\E (X_h)^2&=&\E Y_{M,\gamma,\alpha}^2\leq 2(1-\gamma)^{-M}\leq 2e^{2\gamma M}\leq2e^{2C}.
\end{eqnarray*}
Therefore, since the $X_h$ are pairwise independent and identically distributed,
\begin{eqnarray*}
\P\left(\frac{X_1+\dots+X_K}K\leq C\alpha^2/2\right)&\leq&\P\left(\left|\frac{X_1+\dots+X_K}K-\E X_1\right|\geq C\alpha^2/2\right)\\
&\leq&(C\alpha^2/2)^{-2}\E\left|\frac{X_1+\dots+X_K}K-\E X_1\right|^2\\
&\leq&4\alpha^{-2}C^{-2}K^{-1}\E(X_1-\E X_1)^2\\
&\leq& \frac{8e^{2C}}{\alpha^2C^2K}
\end{eqnarray*}
and for $\gamma$ sufficiently small (by (\ref{wine}), this means $K=\lfloor \gamma/C\varepsilon\rfloor$ sufficiently large), this is less than $\frac12$.
\\
\\Now if we consider the probability space $\Z_{RT}$ with the measure-preserving transformation $\tau x=x+1$, and set $f=f_1+\dots+f_K$ and $X=X_1+\dots+X_K$, then by (\ref{line}), (\ref{joker}) and (\ref{thief}) 
\begin{eqnarray}
\label{worth}
\frac1{\psi(Q_x)}\sum_{k=1}^{\psi(Q_x)} f(x+n_k)> \frac \beta{|\Lambda_{Q_x}|}\sum_{a\in\Lambda_{Q_x}} f(x+a)\geq \beta X(x)\qquad \forall x\not\in E,
\end{eqnarray}
and since $\psi(Q_x)\in\S$, this implies 
$$\P\left(\sup_{N\in\S} \frac1N\sum_{k=1}^N f\circ\tau^{n_k}>CK\alpha^2\beta/2\right)\geq\P\left(\frac XK> \frac{C\alpha^2}{2}\right)-\P(E)\geq\frac14.$$
Therefore 
$$\|\sup_{N\in\S} \frac1N\sum_{k=1}^N f\circ\tau^{n_k}\|_{1,\infty}\geq \frac{CK\alpha^2\beta}{8}\geq \frac{C\alpha^2\beta}{8e^4}\|f\|_1.$$
Since $C$ is arbitrary, there can be no maximal inequality.
\end{proof}
 
\noindent We will prove Proposition \ref{KML} by induction on $K$ and $L$.  We fix $M$ and $\gamma$ at the beginning of the argument (since they will not change as $K$ and $L$ change), and prove each Step $(K,M,L)$ for all values of $A,\delta$ and $D$.
\\
\\Note that Step $(1,M,0)$ is trivial, and that Step $(K-1,M,M)$ implies Step $(K,M,0)$: fix the parameters $A,\delta$ and $D$ and obtain $\{ f_1, \dots, f_{K-1}, X_1,\dots,X_{K-1},E, Q_x\}$ satisfying \textbf{(1)-(4)} on $\Z_T$.  Now set $f_K\equiv 1$, $X_K\equiv 1$ on $\Z$.  This clearly satisfies the conditions.
\\
\\Therefore, to prove Theorem \ref{jimi}, it suffices to show that if $L<M$ and we know Step $(K,M,L)$ and all previous steps (for all values of $A,\delta$ and $D$, and for a fixed $\gamma$), we can prove Step $(K,M,L+1)$ for a fixed $A,\delta$ and $D$ and the same $\gamma$.  We start by applying Step $(K,M,L)$ with $A,\delta/4$ and $D$ to obtain a family $\{T_L,R_L,f_1^L,\dots,f_K^L,X_1^L,\dots,X_K^L,E_L, Q_{x,L}\}$ satisfying \textbf{(1)-(4)}; we must then construct $\{T_{L+1},R_{L+1},f_1^{L+1},\dots,f_K^{L+1},X_1^{L+1},\dots,X_K^{L+1},E_{L+1}, Q_{x,L+1}\}$ to have the required properties.

%Begin The Lemmas Here
\section{Periodic Rearrangements}

\noindent As in \cite{DAAS}, we will essentially construct the next functions by choosing a new $q$ and redefining the current functions on the subset $(-\Lambda_q)^\gamma$.  On this set, we will be selecting a new $Q_{x,L+1}\gg T_L$, and we want to ensure that the left side of (\ref{joker}) will be uniformly large on a significant subset of $(-\Lambda_q)^\gamma$.  Since an average of $f_h^L$ over $\Lambda_Q$ (in the sense of (\ref{joker}) may be irregular for large $Q$, we will modify the $f_h^L$ in advance so that these averages will be bounded below by a constant, while preserving their averages over $\Lambda_Q$ for smaller $Q$.
\\
\\Following Buczolich and Mauldin, we call this modification a \emph{periodic rearrangement}.  Given natural numbers $p\gg T$ with $(p,T)=1$, we will define a linear operator $f\to\tilde f$ from $\ell^1(\Z_T)$ to $\ell^1(\Z_{pT})$ which preserves joint distribution of functions, such that on long blocks each $\tilde f$ is identical to a translate of $f$.
\\
\\In our particular cases, where $\Lambda_T$ consists of the residues of $d$th powers or the integers relatively prime to $T$, we can simply define
\begin{eqnarray*}
\tilde f(x)&:=&\left\{\begin{array}{ll}f(y) & x\in y+p\Z, \;0\leq y< T\cdot\lfloor p/T\rfloor \\
f(x) & \text{otherwise, }\end{array}\right.
\end{eqnarray*}
and prove directly that for every $x\in\Z_{pT}$,
\begin{eqnarray*}
\frac1{|\Lambda_{pT}|}\sum_{a\in\Lambda_{pT}}\tilde f(x+a)\geq\frac12\E f.
\end{eqnarray*}
(For the former case, we would use the P\'{o}lya-Vinogradov inequality on character sums; for the latter we would use Dirichlet's theorem.)
\\ \\ However, we lack such tools when considering more general sequences, so we shall instead use an external randomization in our construction of $\tilde f$. Let $\Omega$ be a probability space and $\xi_i(\omega)$ be independent random variables on $\Omega$, each with a uniform distribution on the discrete set $\{0,\dots, T-1\}$.  Then for any $f\in\ell^1(\Z_T)$ and $\omega\in\Omega$ we define (on the interval $[0,pT)$) the function
\begin{eqnarray}
\label{rearrange}
\tilde f^\omega(x) &:=& \left\{\begin{array}{ll}f(x+\xi_i(\omega)) & (i-1)\lfloor\sqrt p\rfloor T\leq x \leq i\lfloor\sqrt p\rfloor T, 1\leq i \leq\sqrt p; \\
f(x) & (\lfloor\sqrt p \rfloor)^2T\leq x<pT.\end{array}\right.
\end{eqnarray}
Heuristically speaking, we break $\Z_{pT}$ into blocks of size $T\lfloor\sqrt p\rfloor$ and shift each by a random variable $\xi_i(\omega)$.  The point of this is that, although we cannot directly prove that $\Lambda_{pT}$ will be equidistributed in the residue classes modulo $T$, there must exist some value of these shifts under which this is approximately so.
\\ \\
Similarly, for a set $E\subset \Z_T$ we can define $\tilde E^\omega:=\text{supp }\tilde\1_E^\omega\subset \Z_{pT}$.  Note that $\tilde E^\omega$ contains exactly $\lfloor\sqrt p\rfloor |E|$ points in each full block, and $(p-(\lfloor\sqrt p \rfloor)^2)|E|$ points in the last block; thus we see that $\P(\tilde E^\omega)=\P(E)$ for all $E\subset\Z_T$ and all $\omega\in\Omega$, from which it follows that this periodic rearrangement preserves the joint distribution of any collection of functions.
\\ \\For most $x\in \Z_{pT}$ there exists an $\tilde x$ such that $\tilde f^\omega(x+z)= f(\tilde x +z)$ for $0\leq z \ll \sqrt p$.  In particular, if we take our original exceptional set $E\subset\Z_T$, we can define a new exceptional set
\begin{eqnarray}
\label{tildeE}
E^1(\omega):= \tilde E^\omega \cup\bigcup_{0\leq i\leq\frac{\sqrt p}T -1} [iT\lfloor\sqrt p\rfloor -\psi(AT),iT\lfloor\sqrt p\rfloor]+p\Z
\end{eqnarray}
such that (\ref{joker}) and (\ref{thief}) are still satisfied for $\tilde f_h^{L,\omega}$, $\tilde X_h^{L,\omega}$ and $\tilde Q_{x,L}^\omega$ off of $E^1(\omega)$.  Furthermore,
$$\P(E^1(\omega))\leq\P(E)+\frac{\psi(AT)+2T}{T\sqrt p}\leq \delta/2$$
for $p$ sufficiently large.
\\ \\Now we can prove the following lemma:
%Periodic Rearrangement Lemma
\begin{lemma}
\label{wildcat}
For $p$ sufficiently large with $(p,T)=1$, there exists $\omega\in\Omega$ such that for all $0\leq f\in\ell^1(\Z_T)$ and all $x\in\Z_{pT}$,
\begin{eqnarray*}
\frac1{|\Lambda_{pT}|}\sum_{a\in\Lambda_{pT}}\tilde f^\omega(x+a)\geq\frac12\E f.
\end{eqnarray*}
\end{lemma}
\begin{proof}
By the linearity of the periodic rearrangement, it suffices to prove that for some $\omega$, this holds for all characteristic functions of singletons in $\Z_T$: thus it is enough to show that for all $x\in\Z_{pT}$ and $b\in\Z_T$,
\begin{eqnarray*}
\P_\omega\left(\sum_{a\in\Lambda_{pT}}\tilde 1_{\{b\}}^\omega(x+a)<\frac{|\Lambda_{pT}|}{2T}\right)<\frac1{pT^2}.
\end{eqnarray*}
Fix $b$ and $x$; then counting only the main blocks,
\begin{eqnarray*}
\sum_{a\in\Lambda_{pT}}\tilde 1_{\{b\}}^\omega(x+a)&\geq& \sum_{1\leq i\leq\sqrt p}|\{a\in\Lambda_{pT}:(i-1)\lfloor\sqrt p\rfloor T\leq x+a< i\lfloor\sqrt p\rfloor T, \; x+a+\xi_i(\omega)-b \in T\Z\}|\\
&=&\sum_i |\Lambda_{pT}\cap P_i(\omega)|,
\end{eqnarray*}
where $P_i(\omega):=\left\{[(i-1)\lfloor\sqrt p\rfloor T-x,i\lfloor\sqrt p\rfloor T-x)\cap\{b-x-\xi_i+T\Z\}\right\}$ is an arithmetic progression in $\Z_{pT}$.  If we then define
 $$\nu_i(\omega):=p^{-1/2}|\Lambda_{pT}\cap P_i(\omega)|,$$ 
 we see that $0\leq \nu_i\leq1$, that the $\nu_i$ are independent, and that each $z\in \Lambda_{pT}\cap[(i-1)\lfloor\sqrt p\rfloor T-x,i\lfloor\sqrt p\rfloor T-x)$ contributes to $\nu_i(\omega)$ for precisely one value of $\xi_i(\omega)$.  Therefore
\begin{eqnarray*}
\E_\omega\nu_i(\omega)&=& \frac1{T\sqrt p}|\Lambda_{pT}\cap[(i-1)\lfloor\sqrt p\rfloor T-x,i\lfloor\sqrt p\rfloor T-x)|,\\
\Var_\omega\nu_i(\omega)&\leq& \E_\omega\nu^2_i(\omega)\leq \E_\omega\nu_i(\omega).
\end{eqnarray*}
Recall that $\displaystyle \inf_j \P(\Lambda_{p_j})>0$, so that for $p\gg T$, $|\Lambda_{pT}|=|\Lambda_p|\cdot|\Lambda_T|\gg T\sqrt p$.  Thus for $p$ sufficiently large,
\begin{eqnarray*}
\E_\omega(\sum_i\nu_i(\omega))&\geq& \frac1{T\sqrt p}|\Lambda_{pT}\setminus[x-T(p-(\lfloor\sqrt p\rfloor)^2),x]|>\frac{3|\Lambda_{pT}|}{4T\sqrt p}\\
\Var_\omega(\sum_i \nu_i(\omega))&\leq& \frac{|\Lambda_{Tp}|}{T\sqrt p}
\end{eqnarray*}
and we may apply Chernoff's Inequality (Theorem 1.8 from \cite{TV}) to find
\begin{eqnarray*}
\P_\omega\left(\sum_{a\in\Lambda_{pT}}\tilde 1_{\{b\}}^\omega(x+a)<\frac{|\Lambda_{pT}|}{2T}\right)&\leq&\P_\omega\left(\sum_i |\Lambda_{pT}\cap P_i(\omega)|<\frac{|\Lambda_{pT}|}{2T}\right)\\
&=&\P_\omega\left(\sum_i\nu_i(\omega)<\frac{|\Lambda_{pT}|}{2T\sqrt p}\right)\\
&\leq&\P_\omega\left(\sum_i\nu_i(\omega)-\E_\omega \nu_i>\frac{|\Lambda_{pT}|}{4T\sqrt p}\right)\\
&\leq&2\exp(-\frac{|\Lambda_{pT}|}{64T\sqrt p})< \frac1{pT^2}
\end{eqnarray*}
for $p$ sufficiently large (depending on $T$).
\end{proof}
 
\section{Defining $f_K^{L+1}$}
%Definition of f_K^{L+1}
\noindent Using the properties (\ref{wind})-(\ref{earth}), we now take $p=p_{i(K,L)}$ and $q=q_{j(K,L)}$ from our auxiliary sequences such that they are relatively prime to each other and to $T_L$, $R_L$ and $D$, such that $p$ is large enough for Lemma \ref{wildcat} and such that $q\gg D$, $s_q>8\delta^{-1} \psi(AT)$ and
\begin{eqnarray*}
&\P((-\Lambda_q)^\gamma)=\P(\Lambda^\gamma_q)\geq \gamma-\varepsilon,&\\
&\left|\left\{(u,v)\in\Lambda_q\times\Lambda_q:|u-v-w|>\gamma s_q\;\forall w\in\Lambda_q\right\}\right| \geq 5\alpha|\Lambda_q|^2.&
\end{eqnarray*}
We would like to define $f_K^{L+1}$ and $X_K^{L+1}$ as in (\ref{life}) and (\ref{joke}), but then $X_K^{L+1}$ will not precisely equal $Y_{L+1,\gamma,\alpha}$ in distribution.  This is the reason we will make the $X_h^{L+1}$ periodic by $R_{L+1}T_{L+1}$ rather than just $T_{L+1}$: we will later multiply the parts of $X_K^{L+1}$ by the characteristic function of intervals whose lengths are appropriate multiples of $T_{L+1}$.  It will be essential (for its use in later inductive steps) that we keep $(R_{L+1},D)=1$, and for this we will need to define $f_K^{L+1}$ and $X_K^{L+1}$ in a more complicated fashion.
\\ \\ First, we will let
\begin{eqnarray}
\label{alright}
\Phi_q&:=&\left\{u\in -\Lambda_q: \left|\{v\in\Lambda_q:u+v\notin-\Lambda_q+(-\gamma s_q,\gamma s_q)\}\right|\geq2\alpha|\Lambda_q| \right\}
\end{eqnarray}
and note that $|\Phi_q|\geq 3\alpha|\Lambda_q|.$  We then consider $\Phi_q^\gamma=\Phi_q+(0,s_q\gamma)\subset (-\Lambda_q)^\gamma$.  Now $\P(\Phi_q^\gamma)\geq 2\alpha\gamma$ since
\begin{eqnarray*}
\gamma-\varepsilon\leq \P((-\Lambda_q)^\gamma)\leq \P(\Phi_q^\gamma)+q^{-1}\gamma s_q|-\Lambda_q\setminus \Phi_q| \leq \P(\Phi_q^\gamma) + \gamma(1-3\alpha)
\end{eqnarray*}
and we have stipulated that $\varepsilon\leq\alpha\gamma$. \\ \\
Now we choose two sets $\Psi_q,\Delta_q\subset \Z_q$ such that
\begin{eqnarray}
\label{second}
&\Psi_q\subset \Phi_q^\gamma&\\
\label{late}
&\P(\Psi_q)\geq \alpha\gamma&\\
\label{hour}
&\Z_q\setminus (-\Lambda_q)^\gamma \subset \Delta_q \subset \Z_q\setminus \Phi_q&\\
&\P(\Delta_q)<1-\P(\Lambda_q^\gamma)+\delta/8&\\
\label{minute}
&(|\Psi_q|,D)=(|\Delta_q|,D)=1.&
\end{eqnarray}
(As $q\gg D$, this last condition is clearly possible to satisfy simultaneously with the others.)  We observe that
\begin{eqnarray}
\label{falsely}
\left|\{v\in\Lambda_q:x+v\in \Delta_q\}\right|&\geq&2\alpha|\Lambda_q|\qquad\forall x\in\Psi_q,
\end{eqnarray}
and accordingly we define
\begin{eqnarray*}
\label{relief}
f_K^{L+1}&:=&(1-\gamma)^{-1}\tilde f_K^L\1_{\Delta_q}\in\ell^1(\Z_{qpT}).
\end{eqnarray*}
Note that $(q,pT)=1$ implies $\E f_K^{L+1}=(1-\gamma)^{-1}\P(\Delta_q)\E \tilde f_K^L=(1-\gamma)^{-1}\P(\Delta_q)\E f_K^L$; since $\gamma<1/2$ and $\delta<8\varepsilon$, this means
\begin{eqnarray}
\label{growl}
\E f_K^L\leq \E f_K^{L+1}\leq(1+4\varepsilon)\E f_K^L.
\end{eqnarray}
\noindent The goal of Lemma \ref{wildcat} and the definition of $\Psi_q$ is the following lemma:

%Size of X_K^{L+1}
\begin{lemma}
\label{fate}
Let $T=T_L$, with $p$ and $q$ chosen as above.  For all $x\in\Psi_q$ and for any $Q\in\Qset$ such that $qpT\mid Q$,
\begin{eqnarray*}
\frac1{|\Lambda_{Q}|}\sum_{a\in\Lambda_{Q}} f_K^{L+1}(x+a)&\geq&\alpha.
\end{eqnarray*}
\end{lemma}
\begin{proof}
Let $B:=\frac{Q}{qpT}$.  Since $Q\in\Qset$ is squarefree, we see that $B, T, p$ and $q$ are pairwise relatively prime, and
\begin{eqnarray*}
\frac1{|\Lambda_{Q}|}\sum_{a\in\Lambda_{Q}} f_K^{L+1}(x+a)&=&\frac1{|\Lambda_B|}\sum_{w\in\Lambda_B}\frac1{|\Lambda_{qpT}|}\sum_{z\in\Lambda_{qpT}} f_K^{L+1}(x+z)\\
&=&\frac1{|\Lambda_{pT}|}\frac1{|\Lambda_q|}\sum_{u\in\Lambda_{pT}} \sum_{v\in\Lambda_q} (1-\gamma)^{-1}\tilde f_K^L(x+u)\1_{\Delta_q}(x+v)\\
&\geq&\frac {2\alpha}{|\Lambda_{pT}|}\sum_{u\in\Lambda_{pT}}\tilde f_K^L(x+u)\\
&\geq& \alpha \E f_K^L\geq \alpha
\end{eqnarray*}
using (\ref{falsely}) for the first inequality and Lemma \ref{wildcat} for the second.
\end{proof}

\noindent We also define an additional exceptional set
\begin{eqnarray}
\label{hatE}
E_L^2&:=&\{x\in\Delta_q: \exists 0<y\leq \psi(AQ_{x,L})\text{ such that } x+y\notin\Delta_q\}.
\end{eqnarray}
Note that $E_L^2\subset \{\Delta_q\cap(-\Lambda_q)^\gamma\}\cup\{-\Lambda_q+(-\psi(AT),0]\}$.  By our choice of $q$, we see that $$\P(E_L^2)\leq \delta/8+|\Lambda_q|\psi(AT)\leq\frac\delta4.$$  

\section{Restricting a Family to $(-\Lambda_q)^\gamma$}

\noindent As noted in the heuristic outline, if we wish to change $Q_x$ on the set $(-\Lambda_q)^\gamma$, we must change $f_1^L,\dots,f_{K-1}^L$ and $X_1^L,\dots,X_{K-1}^L$ as well, since these need no longer satisfy \textbf{(4)} with the new value of $Q_x$.  In order to find suitable replacement functions on $(-\Lambda_q)^\gamma$, we will use take a Step $(K-1,M,M)$ family $\{S,R', g_1,\dots,g_{K-1},Z_1,\dots,Z_{K-1},E',Q'_x\}$ with suitable parameters, and then restrict the functions $g_h$ to the set $(0,\gamma s_q)+q\Z$ (multiplying them by $|\Lambda_q|$ so that their $\ell^1$ norm is $\gamma$ times its previous value).  The averages along $\Lambda_q$ starting at any $x\in(-\Lambda_q)^\gamma$ will sample from this set uniformly often, so that we may be able to preserve property \textbf{(4)} there.
\\
\\This restriction will not preserve the independence of the $Z_h$ as such, since the averages of the restricted $g_h$ must be 0 off of $(-\Lambda_q)^\gamma$.  However, we may define the conditional expectations $\E(X\mid \Sigma)$; and these will remain independent, for $\Sigma\in(-\Lambda_q)^\gamma$.
\begin{definition}
Let $X$ be a random variable $X$ on a discrete probability space $(\Omega,\P)$, and let $\Sigma\subset\Omega$ with $\P(\Sigma)>0$.  We define the \emph{conditional expectation} $\E(X\mid \Sigma)$ to be the random variable on $\Sigma$ with $\E(X\mid \Sigma)(x)=X(x)\;\forall x\in\Sigma$, where $\Sigma$ is equipped with the probability measure $\P_\Sigma(x)=\P(\Sigma)^{-1}\P(x)$.
\end{definition}
%RESTRICTION TO \Lambda_q
Now we can state the actual form of this restriction for an entire Step $(K,M,L)$ family:
\begin{lemma}
\label{watchtower}
Let $\{T,R,f_1,\dots,f_K,X_1,\dots,X_K,E,Q_x\}$ satisfy the properties of Step $(K,M,L)$ for the parameters $A,\delta,D$.  Say we have $q,B\in\Qset$ with $q,B$ and $T$ pairwise relatively prime, and $qB\leq A$.  Let 
\begin{eqnarray*}
\Xi^\gamma_q&:=&(0,\gamma s_q)+q\Z\\
\bar f_h(x)&:=& |\Lambda_q|f_h(x)\1_{\Xi^\gamma_q}(x)\\
\bar X_h(x)&:=&X_h(x)\1_{(-\Lambda_q)^\gamma}(x)\\
\bar E&:=& E\cap(-\Lambda_q)^\gamma\\
\bar Q_x&:=&qBQ_x\\
\bar A&:=& \frac{A}{qB}\\
\bar D&:=&\frac D{(D,qB)}\\
\bar T&:=& qBT.
\end{eqnarray*}
Then $(\bar T,D)=1$ and
\begin{description}
\item[($\mathbf{\bar1}$)] $\bar f_1,\dots,\bar f_K\in\ell^1(\Z_{\bar T})$ with $\bar f_h\geq0$ and $\E \bar f_h\leq\gamma\E f_h$
\item[($\mathbf{\bar2}$)]  $\bar X_1,\dots,\bar X_K\in\ell^1(\Z_{R\bar T})$ such that for any nonempty $q$-periodic $\Sigma\subset(-\Lambda_q)^\gamma$, $\E(\bar X_1\mid\Sigma),\dots,\E(\bar X_K\mid\Sigma)$ are independent and $\E(\bar X_h\mid\Sigma)\deq X_h$.
\item[($\mathbf{\bar3}$)] $\bar E\subset \Z_{\bar T}$ with $\P(\bar E)\leq \delta \P((-\Lambda_q)^\gamma)$
\item[($\mathbf{\bar4}$)] $\bar Q_x:\Z_{\bar T}\to\N$ such that $\bar Q_x\mid \bar T\;\forall x\in\Z_{\bar T}$ and for each $x\not\in \bar E$,
\begin{eqnarray*}
\frac1{|\Lambda_{\bar Q_x}|}\sum_{a\in\Lambda_{\bar Q_x}} \bar f_h(x+a)\geq \bar X_h(x)&&\forall 1\leq h\leq K
\end{eqnarray*}
and 
\begin{eqnarray*}
\bar f_h(x+y-Q_x)=\bar f_h(x+y)&& \forall1\leq h\leq K,\; \bar Q_x\leq y\leq \psi(\bar A\bar Q_x).
\end{eqnarray*}
\end{description}
\end{lemma}
\begin{proof}
Since $(q,T)=1$, any $q$-periodic set contains an equal portion of integers from each residue class modulo $T$.  This fact quickly implies properties $\mathbf{(\bar1)}$-$\mathbf{(\bar3)}$, noting that the joint distribution of the $X_h$ on any $q$-periodic $\Sigma\subset(-\Lambda_q)^\gamma$ is the same as their joint distribution on $\Z_{RT}$.  Now for property $\mathbf{(\bar4)}$, take $x\in(-\Lambda_q)^\gamma\setminus \bar E$ (since it is trivial otherwise).  Since $Q_x$, $q$, and $B$ are relatively prime,
\begin{eqnarray*}
\frac1{|\Lambda_{\bar Q_x}|}\sum_{a\in\Lambda_{\bar Q_x}} \bar f_h(x+a)&=&\frac1{|\Lambda_B|}\sum_{u\in\Lambda_B}\frac1{|\Lambda_{qQ_x}|}\sum_{z\in \Lambda_{qQ_x}} \bar f_h(x+z)\\
&=& \frac1{|\Lambda_q|}\sum_{v\in \Lambda_q}|\Lambda_q| \1_{\Xi^\gamma_q}(x+v)\cdot\frac1{|\Lambda_{Q_x}|}\sum_{w\in\Lambda_{Q_x}} f_h(x+w)\\
&\geq& \frac1{|\Lambda_{Q_x}|}\sum_{w\in\Lambda_{Q_x}} f_h(x+w) \geq X_h(x)=\bar X_h(x)
\end{eqnarray*}
since if $x\in(-\Lambda_q)^\gamma$, there must exist some $v\in\Lambda_q$ such that $x+v\in\Xi^\gamma_q$.
\\ \\
Finally, $\bar A\bar Q_x= AQ_x$ so (\ref{thief}) implies the last claim trivially.
\end{proof}

%Completion of Inductive Step
\section{Completion of the Inductive Step}

\noindent Now we are ready to define the other functions and prove Step $(K,M,L+1)$; but since so much goes into this step, we will show the origins of the various pieces.
\\
\\ We began with $\{T_L,R_L,f_1^L,\dots,f_K^L,X_1^L,\dots,X_K^L,E_L, Q_{x,L}\}$ satisfying \textbf{(1)-(4)} with the parameters $A,\delta/4,$ and $D$.  We modify these objects in several ways, using a new $p=p_{i(K,L)}$ and $q=q_{j(K,L)}$ chosen in Section 6.
\\
\\ First, we applied the $p$-periodic rearrangement $f\to\tilde f^\omega$ defined in (\ref{rearrange}) to the functions $f_h^L$ and $X_h^L$, with $\omega$ chosen as in Lemma \ref{wildcat}, and defined an associated exceptional set $E_L^1(\omega)$ in (\ref{tildeE}).  We defined sets $\Delta_q, \Psi_q\subset \Z_q$ satisfying (\ref{second})-(\ref{falsely}), then we defined the function $f_K^{L+1}$ in (\ref{relief}) and an additional exceptional set $E_L^2$ in (\ref{hatE}).
\\
\\
We are now ready to proceed.
\\
\\Set $A_L:=AT_Lpq$ and $D_L:=DT_Lpq$.  By our strong inductive hypothesis, we may assume Step $(K-1,M,M)$ and thus construct a family $\{S,R', g_1,\dots,g_{K-1},Z_1,\dots,Z_{K-1},E',Q'_x\}$ satisfying \textbf{(1)-(4)} with parameters $A_L, \delta/4,D_L$.  (In particular, $S$ and $R'$ are each relatively prime to $T_L$, $p$, $q$, and $D$.)
\\
\\ Applying Lemma \ref{watchtower} with $B=Tp$, we obtain $\{\bar S, R', \bar g_1,\dots,\bar g_{K-1},\bar Z_1,\dots,\bar Z_{K-1},\bar E',\bar Q'_x\}$ on $(-\Lambda_q)^\gamma$ satisfying $\mathbf{(\bar1)}$-$\mathbf{(\bar4)}$ with the parameters $\bar A_L=A,\delta/4,$ and $\bar D_L=D$.
\\
%Definition of f_h^{L+1}, X_h^{L+1}
\\ Then we define
\begin{eqnarray}
T_{L+1}&:=&\bar S=ST_Lpq\\
E_{L+1}&:=& E_L^1(\omega)\cup E_L^2\cup\bar E'\\
Q_{x,L+1}&:=&\left\{\begin{array}{ll}\tilde Q_{x,L},& x\in\Delta_q \\ \bar Q'_x, & x\not\in\Delta_q,\end{array}\right.
\end{eqnarray}
and for $1\leq h\leq {K-1}$,
\begin{eqnarray}
f_h^{L+1}&:=&\tilde f_h^L\1_{\Delta_q}+\bar g_h\1_{\Z\setminus\Delta_q}\in\ell^1(\Z_{T_{L+1}})\\
X_h^{L+1}&:=&\tilde X_h^L\1_{\Delta_q}+\bar Z_h\1_{\Z\setminus\Delta_q}\in\ell^1(\Z_{T_{L+1}}).
\end{eqnarray}

%Definition of X_K^{L+1}
\noindent We have already defined
\begin{eqnarray*}
f_K^{L+1}&:=&(1-\gamma)^{-1}\tilde f_K^L\1_{\Delta_q}\in\ell^1(\Z_{T_{L+1}}).
\end{eqnarray*}
\noindent It thus remains to define $X_K^{L+1}$.  As noted before, we cannot simply define it as in (\ref{joke}); we must reduce it slightly so that it equals $Y_{L+1,\gamma,\alpha}$ in distribution.
\\
\\Recall that $\P(\Delta_q)\geq1-\gamma$ and $\P(\Psi_q)> \alpha\gamma$; we have taken $\gamma, \alpha$ to be dyadic rationals and assumed (\ref{minute}), so we may write 
\begin{eqnarray}
\label{soul}
\frac{1-\gamma}{\P(\Delta_q)}=\frac sr,\qquad\frac{\alpha\gamma}{\P(\Psi_q)}=\frac tr
\end{eqnarray}
with $(r,D)=1$.  (Recall that $D$ is odd.)  Everything so far is periodic with period $T_{L+1}R_LR'$, so if we define 
\begin{eqnarray*}
R_{L+1}&:=&R_LR'r\\
\Gamma_s&:=& [0,sR'R_LT_{L+1})+R_{L+1}T_{L+1}\Z\\
\Gamma_t&:=& [0,tR'R_LT_{L+1})+R_{L+1}T_{L+1}\Z\\
X_K^{L+1}&:=& (1-\gamma)^{-1}\tilde X_K^L\1_{\Delta_q}\1_{\Gamma_s}+ \alpha\1_{\Psi_q}\1_{\Gamma_t},
\end{eqnarray*}
then $X_K^{L+1}\deq Y_{L+1,\gamma,\alpha}$ and all of the $X_h^{L+1}$ are periodic with period $R_{L+1}T_{L+1}$.
\\
\\ We now have a family $\{T_{L+1},R_{L+1},f_1^{L+1},\dots,f_K^{L+1},X_1^{L+1},\dots,X_K^{L+1},E_{L+1}, Q_{x,L+1}\}$, with $T_{L+1}\in\Qset$ and $(T_{L+1},D)=(R_{L+1},D)=1$.  We must check the four properties of Step $(K,M,L+1)$:
%PROOF OF PROPERTIES
\begin{description}
\item[(1)] $f_1^{L+1}, \dots,f_K^{L+1}\in\ell^1(\Z_{T_{L+1}})$ with $f_h^{L+1}\geq0$ and $1\leq\E f_h^{L+1}\leq(1+4\varepsilon)^{(K-1)M+L+1}$.
\begin{proof}
For $1\leq h\leq {K-1}$, by the inductive hypothesis we see
\begin{eqnarray*}
\E f_h^{L+1}&=&\E \tilde f_h^L\P(\Delta_q)+\E \bar g_h^L\\
&\leq& (1-\gamma+\varepsilon)\E f_h^L+\gamma \E g_h^L\\
&\leq& (1+\varepsilon)(1+4\varepsilon)^{(K-1)M+L}\\
&\leq& (1+4\varepsilon)^{(K-1)M+L+1}.
\end{eqnarray*}
For $f_K^{L+1}$, this follows from (\ref{growl}) and the inductive hypothesis.
\end{proof}
\item[(2)] $X_1^{L+1},\dots,X_K^{L+1}\in\ell^1(\Z_{R_{L+1}T_{L+1}})$ pairwise independent with $X_h^{L+1}\deq Y_{M,\gamma,\alpha}$ for $1\leq h\leq {K-1}$, and $X_K^{L+1}\deq Y_{L+1,\gamma,\alpha}$.
\begin{proof}
We have inductively assumed that $X_1^L,\dots,X_K^L$ are pairwise independent; since the periodic rearrangement preserves joint distribution, this is true of $\tilde X_1^L,\dots,\tilde X_K^L$ as well.  We begin by considering the conditional expectations of the $\tilde X_h^L$ on $\Delta_q$.  Recall that $q$ is relatively prime to $p$, $T_L$, $R_L$, $S$, and $R'$, so that $\Delta_q$ contains an equal proportion of all residue classes modulo $pT_LR_LSR'$.  Furthermore, $\Gamma_s$ and $\Gamma_t$ each contain an equal proportion of all residue classes modulo $R'R_LT_{L+1}$.  Thus for any $h< K$ and $\lambda_h,\lambda_K>0$,
\begin{eqnarray*}
\P(x\in\Delta_q: X_h^{L+1}(x)\geq\lambda_h, X_K^{L+1}(x)\geq\lambda_K)&=&\P(x\in\Delta_q\cap \Gamma_s,\tilde X_h^L(x)\geq\lambda_h, \tilde X_K^L(x)\geq\lambda_K)\\
&=& \frac sr\P(\Delta_q)\P(\tilde X_h^L(x)\geq\lambda_h)\P(\tilde X_K^L(x)\geq\lambda_K)\\
&=&\P(\Delta_q)\P(X_h^{L+1}(x)\geq\lambda_h)\P(X_K^{L+1}(x)\geq\lambda_K).
\end{eqnarray*}
Thus $\E(X_h^{L+1}\mid \Delta_q)$ and $\E(X_K^{L+1}\mid \Delta_q)$ are independent; similarly, $\E(X_h^{L+1}\mid \Delta_q)$ and $\E(X_{h'}^{L+1}\mid \Delta_q)$ are independent for any $h<h'<K$.
\\
\\We proceed similarly on the rest of $\Z$, letting $\Sigma$ denote either $\Psi_q$ or $\Z_{T_{L+1}}\setminus(\Delta_q\cup\Psi_q)$; on each of these, Lemma \ref{watchtower} implies that $\E(\bar Z_1^L\mid\Sigma),\dots,\E(\bar Z_{K-1}^L\mid\Sigma)$ are pairwise independent and distributed like $Y_{M,\gamma,\alpha}$.  Clearly $\E(\1_{\Gamma_t}\mid \Sigma)$ is independent of any of these, so $\E(X_1^{L+1}\mid\Sigma),\dots,\E(X_K^{L+1}\mid\Sigma)$ are pairwise independent for $\Sigma=\Psi_q$ or $\Z_{T_{L+1}}\setminus(\Delta_q\cup\Psi_q)$.
\\
\\Now for each $1\leq h< K$, 
\begin{eqnarray}
\label{norwegian}
\E(X_h^{L+1}\mid\Delta_q)\deq\E(X_h^{L+1}\mid\Psi_q)\deq \E(X_h^{L+1}\mid \Z_{T_{L+1}}\setminus(\Delta_q\cup\Psi_q))\deq X_h^{L+1}\deq Y_{M,\gamma,\alpha},
\end{eqnarray}
 and so for any $1\leq h < h'\leq K$ and $\lambda_h,\lambda_{h'}>0$,
\begin{eqnarray*}
\P(X_h^{L+1}(x)\geq\lambda_h, X_{h'}^{L+1}(x)\geq\lambda_{h'})&=&\sum_\Sigma \P(x\in\Sigma: X_h^{L+1}(x)\geq\lambda_h, X_{h'}^{L+1}(x)\geq\lambda_{h'})\\
&=&\sum_\Sigma \P(\E(X_h^{L+1}\mid\Sigma)(x)\geq\lambda_h, \E(X_{h'}^{L+1}\mid\Sigma)(x)\geq\lambda_{h'})\P(\Sigma)\\
&=&\sum_\Sigma \P(\E(X_h^{L+1}\mid\Sigma)(x)\geq\lambda_h)\P(\E(X_{h'}^{L+1}\mid\Sigma)(x)\geq\lambda_{h'})\P(\Sigma)\\
&=&\sum_\Sigma \P(X_h^{L+1}(x)\geq\lambda_h)\P(\E(X_{h'}^{L+1}\mid\Sigma)(x)\geq\lambda_{h'})\P(\Sigma)\\
&=&\P(X_h^{L+1}(x)\geq\lambda_h)\sum_\Sigma \P(x\in\Sigma: X_{h'}^{L+1}(x)\geq\lambda_{h'})\\
&=&\P(X_h^{L+1}(x)\geq\lambda_h)\P( X_{h'}^{L+1}(x)\geq\lambda_{h'}).
\end{eqnarray*}
(The sum is over the sets $\Sigma=\Delta_q,\Psi_q,$ and $\Z_{T_{L+1}}\setminus(\Delta_q\cup\Psi_q)$; the property (\ref{norwegian}) enters in at the fourth equality.)  Thus we have preserved independence.
\\ \\ We have already noted that $X_K^{L+1}\deq Y_{L+1,\gamma,\alpha}$.
\end{proof}
\item[(3)] An exceptional set $E_{L+1}\subset \Z_{T_{L+1}}$ with $\P(E_{L+1})\leq \delta$.
\begin{proof}
$\P(E_{L+1})\leq\P(E_L^1(\omega))+\P( E_L^2)+\P(\bar E'_L)
\leq\delta/2+\delta/4+\delta\gamma/4\leq\delta.$
\end{proof}
\item[(4)] $Q_x=Q_{x,L+1}:\Z_{T_{L+1}}\to\Z^+$ with $Q_x\in\Qset$ and $Q_x\mid T_{L+1}\;\forall x$, such that for each $x\not\in E_{L+1}$,
\begin{eqnarray*}
\frac1{|\Lambda_{Q_x}|}\sum_{a\in\Lambda_{Q_x}} f_h^{L+1}(x+a)\geq X_h^{L+1}(x)&&\forall 1\leq h\leq K
\end{eqnarray*}
and
\begin{eqnarray*}
f_h(x+y-Q_x)=f_h(x+y)\qquad \forall1\leq h\leq K,\; Q_x\leq y\leq \psi(AQ_x).
\end{eqnarray*}
\begin{proof}
Since $\tilde Q_{x,L}\mid T_L$ and $\bar Q'_x\mid \bar S$, clearly $Q_{x,L+1}\mid T_{L+1}$; in addition, $\tilde Q_{x,L}\in\Qset$ and $\bar Q'_x=pqT_LQ'_x\in\Qset$ (this is squarefree since $p,q,T_L,$ and $Q'_x$ are pairwise relatively prime).
\\
\\For $x\in\Delta_q\setminus E_{L+1}$ and $1\leq y\leq \psi(AT_L)$, we see that $$f_h^{L+1}(x+y)=\tilde f_h^L(x+y)=f_h^L(x+y+\xi_i(\omega))$$ for some $\omega$ fixed and $i$ depending only on $x$.  Thus \textbf{(4)} follows from the previous step.
\\ \\On $\Psi_q$ and $\Z\setminus(\Delta_q\cup\Psi_q)$, since $pqT_L \mid \bar Q'_x$, this is just Lemma \ref{watchtower} for $h<K$, and for $h=K$ this is Lemma \ref{fate} combined with the observation that $f_K^{L+1}$ is periodic by $pqT_L$.
\end{proof}
\end{description}
Thus we have proved Step $(K,M,L+1)$; by induction, we have proved Proposition \ref{KML} and Theorem \ref{jimi}.

%NOTES ON THE PROOF
\section{Notes on the Proof}

Several of the conditions, parameters and lemmas in this complicated argument appear on a first reading to be extraneous to the proof.  In the interest of clarity, we find it helpful to outline in hindsight the purposes of the following:

\begin{itemize}
\item The condition (\ref{wine}) lets us prove Theorem \ref{jimi} from Proposition \ref{KML} (note that this proof requires $\gamma/C\varepsilon_\gamma \to\infty$ as $\gamma\to0$).  (\ref{earth}) comes in at (\ref{alright}) and the subsequent definition of $\Psi_q$, and (\ref{line}) allows us to claim (\ref{worth}).
\item Prescribing an exact distribution $Y_{n,\gamma,\alpha}$ (defined in (\ref{distance})) for the $X_h^L$ is necessary in order to guarantee (\ref{norwegian}), which ensures that pairwise independence is preserved.
\item $\gamma$ and $\alpha$ must be dyadic rationals, and $\Delta_q$ and $\Psi_q$ must be chosen to satisfy (\ref{minute}), so that we can assume $(r,D)=1$ in (\ref{soul}), so that we can have $(R,D)=1$ in Step $(K,M,L)$, so that we can choose our $(K-1,M,M)$ family with $(R',q)=1$, so that $X_K^{L+1}$ will be independent of the other $X_h^{L+1}$.
\item The parameter $D$ lets us guarantee that when we inductively introduce a Step $(K-1,M,M)$ family, we can ensure that its period $R'S$ is relatively prime to $T_L,p,$ and $q$, thus allowing us to apply Lemma \ref{watchtower}.
\item We have two distinct parameters $T$ and $R$ because we will need the period of $f_K^L$ to be squarefree in Lemma \ref{fate} (because the result of \cite{Poisson} only applies for squarefree moduli), but the operation of reducing $X_K^{L+1}$ to its proper distribution will multiply its period by a large power of 2.
\item The condition (\ref{thief}) is necessary for (\ref{worth}), connecting the actual averages over the sequence $\{n_k\}$ with the averages over a set of residues $|\Lambda_Q|\in\Z_Q$.  The parameter $A$ must be allowed to take arbitrarily large values, although it need only be $\geq 1$ when used in (\ref{worth}), because each application of Lemma \ref{watchtower} divides it by a large constant.
\item The periodic rearrangement defined in (\ref{rearrange}) puts a uniform lower bound on the averages of $\tilde f_K^L$ over $Q_{x,L+1}$ on a set which depends only on $q$ and not on $f_K^L$; this allows us to prove Lemma \ref{fate}.
\end{itemize}

\noindent The author thanks his dissertation advisor, M. Christ, for several key suggestions including the external randomization, and M. Wierdl for many encouraging and stimulating discussions on this topic.

\bibliography{LaVictoire}{}
\bibliographystyle{plain}

\noindent \textsc{Patrick LaVictoire\\
Department of Mathematics, UC Berkeley \\  Berkeley, CA 94720-3840 USA}\\
\emph{E-mail:} patlavic@math.berkeley.edu\\
\emph{URL:} http://math.berkeley.edu/$\sim$patlavic/

\end{document}